\documentclass[a4paper,11pt]{article}
\usepackage{amssymb,amsmath,amsthm, mathrsfs,chngcntr}
\usepackage[pdftex]{graphicx}

\numberwithin{equation}{section}
\newtheorem{thm}{Theorem}
\newtheorem{cor}[thm]{Corollary}
\newtheorem{lemma}[thm]{Lemma}
\newtheorem{prop}[thm]{Proposition}
\newtheorem*{conj}{Conjecture}

\theoremstyle{definition}

\theoremstyle{definition}
\newtheorem{remark}{Remark}

\newcommand{\C}{\mathbb{C}}
\newcommand{\R}{\mathbb{R}}
\newcommand{\Z}{\mathbb{Z}}

\newcommand{\ud}{\mathrm{d}}
\newcommand{\Sc}{\mathcal{S}}

\counterwithout{equation}{subsection}

\author{ Jonathan Hickman }

\title{An affine Fourier restriction theorem for conical surfaces}

\begin{document}

\maketitle

\begin{abstract}

A Fourier restriction estimate is obtained for a broad class of conic surfaces by adding a weight to the usual underlying measure. The new restriction estimate exhibits a certain affine-invariance and implies the sharp $L^p - L^q$ restriction theorem for compact subsets of a type $k$ conical surface, up to an endpoint. Furthermore, the chosen weight is shown to be, in some quantitative sense, optimal. Appended is a discussion of type $k$ conical restriction theorems which addresses some anomalies present in the existing literature.

\end{abstract}

\paragraph{0. Introduction.} In \cite{Nicola2009} Nicola gave an alternative proof of the sharp $L^p - L^q$ Fourier restriction theorem for the conical surface $\{(\xi, |\xi|) : \xi \in \widehat{\R}^2\}$ lying in the frequency space $\widehat{\R}^3$ with measure $\ud \xi/ |\xi|$, originally due to Barcel\'o \cite{Barcelo1985}. Explicitly, this states whenever $1 \leq p <4/3$ and $q = p'/3$ one has
\begin{equation*}
\bigg( \int_{\widehat{\R}^2} | \hat{F}(\xi, |\xi|)|^q \,\frac{\ud\xi}{|\xi|} \bigg)^{1/q} \leq A_p \|F\|_{L^p(\R^3)} \qquad \textrm{for all $F \in \Sc(\R^3)$}
\end{equation*}
where $A_p >0$ is some constant depending on $p$ only and $\Sc(\R^n)$ denotes the space of Schwartz functions. Here it is observed Nicola's arguments can easily be adapted to give results in an affine-invariant setting. In particular, it is shown that Sj\"olin's affine restriction estimate for convex plane curves \cite{Sjolin1974} implies a variant of the conical restriction theorem where one may replace the circular cone with any member of a broad class of conic surfaces given by dilating convex curves, provided the measure $\ud \xi / |\xi|$ is substituted with a suitable measure which is affine-invariant on each `slice' of the conical surface. To make this discussion precise, let $\Sigma \subset \widehat{\R}^2$ be a smooth plane curve, given by the boundary of some centred convex body. That is, $\Sigma$ equals $\partial \Omega$ where $\Omega$ is compact, convex with smooth boundary for which $0 \in \Omega$ is an interior point. There exists a function $\phi : \widehat{\R}^2 \rightarrow [0, \infty)$ that is smooth away from the origin, homogeneous of degree 1 and satisfies $\phi(\xi) =1$ if and only if $\xi \in \Sigma$. Define the conical surface $\mathcal{C}$ generated by $\Sigma$ by
\begin{equation*}
\mathcal{C} := \big\{ (\xi, \phi(\xi)) : \xi \in \widehat{\R}^2 \big\} \subset \widehat{\R}^3.
\end{equation*}
Of course, the prototypical example is given by taking $\Sigma = S^1$ the unit circle, in which case $\phi(\xi) = |\xi|$ and $\mathcal{C}$ is the usual `circular' cone.

Initially consider the restriction problem for the compact piece of cone $S:= \{(\xi, \phi(\xi)) : \xi \in \Delta \}$ where $\Delta := \{ \xi \in \widehat{\R}^2 : 1 \leq \phi(\xi) \leq 2 \}$ endowed with surface measure $\ud \sigma$. One wishes to determine the exponents $(p,q)$ for which there exists some constant $A_{p,q, S} > 0$ (which may depend on the choice of exponents and on the choice of cone) such the following estimate holds:
\begin{equation}\label{conicr}
\|\hat{F}|_{S}\|_{L^q(S, \ud \sigma)} \leq A_{p,q, S} \|F\|_{L^p(\R^3)}\qquad \textrm{for all $F \in \Sc(\R^3)$}.
\end{equation}

The expected range of $p$ and $q$ is determined by properties of the curvature of $\Sigma$ and Knapp-type counter-examples. Sharp results in the non-degenerate case where the curve $\Sigma$ is assumed to possess everywhere non-vanishing curvature were established in \cite{Barcelo1985, Nicola2009} (see also \cite{Drury1993, Oberlin2002}). More generally, one can admit possible degeneracies provided the order of contact of the curve with any tangent line is everywhere bounded by some fixed integer $k \geq 2$. In this case one says the curve is of type $k$, noting that $k=2$ corresponds to the non-degenerate case. The inequality \eqref{conicr} was discussed for type $k$ conical surfaces in \cite{Barcelo1986} where it was established for a sub-optimal range of $p, q$ and, more recently, an improvement appeared in \cite{Buschenhenke2012} (for further discussion, see the appendix).

Related to estimates for a compact piece of the cone are restriction theorems for the whole cone $\mathcal{C}$, such as the theorem stated at the start of the introduction. In this case $\mathcal{C}$ is equipped with a measure given by weighting the surface measure with a negative power of $\phi(\xi)$ so that the problem behaves well under scaling (in the above example the measure is $\ud \xi / |\xi|$; for more general examples see, for instance, \cite{Barcelo1986}).  

This paper considers conical restriction theorems not only in the type $k$ case, but admits the possibility of points where $\Sigma$ is completely flat. Interesting results can be obtained in this more general setting by weighting the measure by a suitable function which vanishes at the degenerate points. The weight ameliorates the effect of these degeneracies and one can hope to achieve $L^p - L^q$ boundedness for the full range of exponents corresponding to the non-degenerate case. This strategy follows the example of numerous authors (notably Sj\"olin \cite{Sjolin1974} and Drury \cite{Drury1990}) who, in considering Fourier restriction problems involving degenerate curves or surfaces, have replaced the underlying surface measure $\ud \sigma$ with affine surface measure $\kappa^{1/(n+1)}\ud \sigma$. This measure has the desired effect of dampening any degeneracies of the curve or surface and also makes the problem both affine and parametrisation invariant. 

When considering conic surfaces the affine surface measure is not suitable (it is the zero measure). However, here a variant of this measure is described which leads to interesting restriction results exhibiting certain affine invariance properties. 

In the following section the weighted restriction theorem is stated and proved. It is also shown to imply some of the finite-type results mentioned above. Higher dimensional analogues are stated and shown to follow from the (open) affine restriction conjecture. The latter section demonstrates that the choice of weight is optimal in some quantitative sense. Appended is a discussion of type $k$ conical restriction theorems in which some discrepancies existing in the literature are clarified. 

The author wishes to thank and acknowledge his PhD supervisor, Prof. Jim Wright, for all his kind help and guidance on this work.

\paragraph{1. Weighted restriction theorem.} Let $\Sigma$, $\phi$ and $\mathcal{C}$ be as in the introduction. Define the weight function
\begin{equation*}
w(\xi):= \big\langle M(\phi)(\xi) \nabla \phi(\xi) , \nabla \phi(\xi) \big\rangle \phi(\xi).
\end{equation*}
Here $\langle \, \cdot \, , \, \cdot \, \rangle$ denotes the Euclidean inner product and $M(\phi)$ is the matrix-valued function 
\begin{equation*}
M(\phi) := \left( \begin{array}{cc} -\frac{\partial^2 \phi}{\partial \xi_1^2} & \frac{\partial^2 \phi}{\partial \xi_1\partial \xi_2} \\
\frac{\partial^2 \phi}{\partial \xi_1\partial \xi_2} & -\frac{\partial^2 \phi}{\partial \xi_2^2} \end{array} \right).
\end{equation*}
Notice $M(\phi)$ is the negative of the adjugate of the Hessian matrix of $\phi$. One may easily verify $w$ is smooth away from the origin and homogenous of degree 0. The desired restriction estimate for the whole cone is as follows:

\begin{prop}\label{affrest} For $\phi$ and $w$ as above, if $1 \leq p < 4/3$ and $q = p'/3$ then
\begin{equation}\label{afconic}
\bigg(\int_{\widehat{\R}^2} |\widehat{F}(\xi, \phi(\xi))|^q w(\xi)^{1/3} \,\frac{\ud \xi}{\phi(\xi)}\bigg)^{1/q} \leq A_p \|F\|_{L^p(\R^3)}  
\end{equation}
for all $F \in \Sc(\R^3)$. Here $A_p$ is a universal constant in the sense that it depends on $p$ only and, in particular, not the choice of conical surface. 
\end{prop}

\begin{remark}
\begin{itemize}
\item[a)] In the prototype case $\Sigma = S^1$, $w(\xi) = 1$ and the original restriction estimate \eqref{conicr} is recovered.
\item[b)] The inequality \eqref{afconic} exhibits certain kind of affine invariance. If \eqref{afconic} holds for a fixed $\phi$ then it is easily seen to hold \emph{with the same constant} $A_p$ whenever $\phi$ replaced with any function of the form $\phi \circ X$ where $X \in \mathrm{GL}(2, \R)$ is an invertible linear transformation.
\end{itemize}
\end{remark}

As indicated in the introduction, the proof of the proposition is given simply by observing that the arguments of Nicola in \cite{Nicola2009} may be adapted to work in this setting. The exposition will therefore be terse; the reader is directed to the aforementioned paper \cite{Nicola2009} for further details.

Before giving the proof some preliminary remarks are in order. For each $t \geq 0$ let $\Sigma_t := t\Sigma$ denote the $t$-dilate of $\Sigma$ so that the cone $\mathcal{C}$ may be expressed as a disjoint union of a continuum of slices:
\begin{equation*}
\mathcal{C} = \bigcup_{t \geq 0} \Sigma_t \times \{t\}.
\end{equation*}
Let $\ud \sigma_t$ denote the surface measure on $\Sigma_t$ with $\ud \sigma := \ud\sigma_1$ and $\kappa$ the curvature of $\Sigma$. Elementary differential geometry yields the following identity:
\begin{equation*}
\kappa(\xi) = \frac{w(\xi)}{|\nabla \phi (\xi) |^3} \qquad \textrm{for all $\xi \in \Sigma$.}
\end{equation*}

A theorem of Sj\"olin \cite{Sjolin1974} (see also \cite{Oberlin2001}) implies for $1 \leq p < 4/3$ and $q = p'/3$ one has
\begin{equation}\label{Sjolin}
\| \hat{f}|_{\Sigma}\|_{L^q(\Sigma, \,\kappa^{1/3}\ud \sigma)} \leq A_p \|f\|_{L^p(\R^2)} \qquad \textrm{for all $f \in \Sc(\R^2)$}
\end{equation}
where the constant $A_p$ is independent of the choice of convex curve $\Sigma$. It will be shown that the conic restriction estimate is related to \eqref{Sjolin} via the \emph{co-area formula}
\begin{equation}\label{coarea}
\int_{\widehat{\R}^2} g(\xi) \,\ud \xi = \int_0^{\infty} \int_{\Sigma_t} g(\xi) \frac{\ud \sigma_t(\xi)}{|\nabla \phi(\xi)|} \,\ud t,
\end{equation}
valid for all non-negative continuous functions $g$ on $\widehat{\R}^2$. For a proof of this identity see \cite{Federer2008}. 

Finally, a word on notation. Throughout this section, for a list of objects $L$ and $X, Y \geq 0$ the notation $X \lesssim_L Y$ signifies $X \leq A_L Y$ where $A_L > 0$ denotes a constant depending only on the objects appearing in $L$. The relation $X \gtrsim_L Y$ is defined in a similar fashion and $X \sim_L Y$ is taken to mean $X \lesssim_L Y \lesssim_L X$.

\begin{proof}[Proof (of Proposition \ref{affrest})] Fixing exponents $p, q$ satisfying the hypotheses of the proposition, it suffices to establish the dual \emph{extension} estimate
\begin{equation}\label{exafc}
\|(u \ud \mu_{\mathcal{C}})\,\check{}\,\|_{L^{p'}(\R^3)} \lesssim_p \|u\|_{L^{q'}(\mathcal{C}, \ud\mu_{\mathcal{C}})}
\end{equation}
where $\ud \mu_{\mathcal{C}}$ denotes the weighted conic measure so that

\begin{equation*}
(u \ud \mu_{\mathcal{C}})\,\check{}\,(x,t) = \int_{\widehat{\R}^2} e^{  2\pi i(x.\xi+t\phi(\xi))} u(\xi, \phi(\xi)) w(\xi)^{1/3} \,\frac{\ud \xi}{\phi(\xi)}.
\end{equation*}
By applying the co-area formula together with a change of variables one obtains
\begin{eqnarray*}
(u \ud \mu_{\mathcal{C}})\,\check{}\,(x,t)&=& \int_0^{\infty} \int_{\Sigma_s} e^{2\pi i(x.\xi+ts)} u(\xi, s) \frac{w(\xi)^{1/3}}{|\nabla \phi (\xi)|} \,\ud \sigma_s (\xi)  \,\frac{\ud s}{s} \\
&=& \int_0^{\infty} e^{2\pi its} \int_{\Sigma} e^{2\pi isx.\xi'} u(s\xi',s) \kappa(\xi')^{1/3} \,\ud \sigma (\xi') \,\ud s.
\end{eqnarray*}
Notice that the last integral is the value at $t$ of the inverse Fourier transform of the function
\begin{equation*}
\chi_{[0, \infty)}(s) \int_{\Sigma} e^{2\pi isx.\xi'} u(s\xi',s) \kappa(\xi')^{1/3} \,\ud \sigma (\xi').
\end{equation*}
Apply the Lorentz space version of the Hausdorff-Young inequality to obtain
\begin{eqnarray*}
\|(u \ud \mu_{\mathcal{C}})\,\check{}\,\|_{L^{p'}_{x,t}(\R^3)} &\lesssim_p& \bigg\|\int_{\Sigma} e^{2\pi isx.\xi'} u(s\xi',s) \kappa(\xi')^{1/3} \,\ud \sigma (\xi')\bigg\|_{L^{p'}_xL_s^{p,p'}(\R^2 \times \R_+)} \\
&\lesssim_p& \bigg\|\int_{\Sigma} e^{2\pi isx.\xi'} u(s\xi',s) \kappa(\xi')^{1/3} \,\ud \sigma (\xi')\bigg\|_{L_s^{p,p'}L^{p'}_x(\R^2 \times \R_+)}
\end{eqnarray*}
where the second inequality is due to the interchange lemma from \cite{Nicola2009}. By a change of variables and an appeal to the dual formulation of the estimate \eqref{Sjolin} one deduces
\begin{eqnarray*}
\bigg\|\int_{\Sigma} e^{2\pi isx.\xi'} u(s\xi',s) \kappa(\xi')^{1/3} \,\ud \sigma (\xi')\bigg\|_{L^{p'}_x(\R^2)}
&=& s^{-2/p'} \|(u(s \,\cdot \,,s)\kappa(\,\cdot\,)^{1/3} \,\ud \sigma)\,\check{}\,\|_{L^{p'}_x(\R^2)} \\
&\lesssim_p & s^{-2/p'} \|u(s \,\cdot \,,s)\|_{L^{q'}_{\xi}(\Sigma, \kappa^{1/3} \,\ud \sigma)}.
\end{eqnarray*}
Observe the hypotheses on the exponents imply $q' \leq p'$ and $1/p - 1/q' = 2/p'$. Thus, by the nesting of Lorentz spaces and Lorentz version of H\"older's inequality,
\begin{eqnarray*}
\|(u \ud \mu_{\mathcal{C}})\,\check{}\,\|_{L^{p'}_{x,t}(\R^3)} &\lesssim_p& \big\|s^{-2/p'} \|u(s \,\cdot \,,s)\|_{L^{q'}_{\xi}(\Sigma, \kappa^{1/3} \,\ud \sigma)}\big\|_{L_s^{p,p'}(\R_+)} \\
 &\lesssim_p& \big\|s^{-2/p'} \|u(s \,\cdot \,,s)\|_{L^{q'}_{\xi}(\Sigma, \kappa^{1/3} \,\ud \sigma)}\big\|_{L_s^{p,q'}(\R_+)} \\
&\lesssim_p& \|s^{-2/p'}\|_{L_s^{p'/2,\infty}(\R_+)} \big\| \|u(s \,\cdot \,,s)\|_{L^{q'}_{\xi}(\Sigma, \kappa^{1/3} \,\ud \sigma)}\big\|_{L_s^{q'}(\R_+)}.
\end{eqnarray*}
Finally recall
\begin{equation*}
\|s^{-2/p'}\|_{L_s^{p'/2,\infty}(\R_+)} = \sup_{\alpha > 0} \alpha|\{s > 0 : s^{-2/p'} > \alpha\}|^{2/p'} = 1
\end{equation*}
whilst by an easy computation, essentially a reversal of the identities used at the start of the proof, one deduces
\begin{equation*}
 \| \|u(s \,\cdot \,,s)\|_{L^{q'}_{\xi}(\Sigma,\kappa^{1/3} \,\ud \sigma)}\|_{L_s^{q'}(\R_+)}  = \|u\|_{L^{q'}(\mathcal{C}, \ud \mu_{\mathcal{C}})}
\end{equation*}
and thence the required estimate.
\end{proof}

By applying H\"older's inequality one obtains a sharp restriction theorem for a compact piece of the cone. In particular, consider restriction to the surface $S := \{(\xi, \phi(\xi)) : \xi \in \Delta\}$ where $\Delta := \{\xi \in \widehat{\R}^2 : 1 \leq \phi(\xi) \leq 2\}$. 

\begin{cor}\label{cor1} For $1 \leq p < 4/3$ and $1 \leq q \leq p'/3$,
\begin{equation*}
\bigg(\int_{\Delta} |\widehat{F}(\xi, \phi(\xi))|^q w(\xi)^{1/3} \,\ud \xi\bigg)^{1/q} \leq A_{p,q, S} \|F\|_{L^p(\R^3)}  
\end{equation*}
for all $F \in \Sc(\R^3)$.
\end{cor}

It is interesting to note some applications of the preceding weighted restriction inequalities to the unweighted theory. Clearly if the curvature of $\Sigma$ is non-vanishing then $w(\xi)$ is bounded below by some positive constant. Thus the weighted results imply both the sharp restriction theorem for the compact piece of the cone with surface measure and for the whole cone with the scale-invariant measure in the non-degenerate case. 

The proposition can also be used to obtain the results when $\Sigma$ is of finite type and gives the sharp range of exponents, except for an endpoint. First note for sub-critical exponents $1 \leq p < \tfrac{k+1}{k}$ and $\tfrac{k+1}{p'} < \frac{1}{q}$ where $k\geq 3$ is the type of $\Sigma$, the estimate
\begin{equation*}
\|\hat{F}|_{S}\|_{L^q(S, \ud \sigma)} \leq A_{p,q, S} \|F\|_{L^p(\R^3)}  
\end{equation*}
is a simple consequence of the previous corollary and H\"older's inequality. One may also obtain results on the critical line $\tfrac{k+1}{p'} = \frac{1}{q}$ by applying a simple interpolation argument, of the type described in \cite[Remark 2.2]{Bak2011}. 

\begin{cor}\label{cor2} Suppose $\Sigma$ is of finite type, let $S$ be as above and $\ud \sigma$ denote surface measure on $S$. For $1 \leq p < \tfrac{k+2}{k+1}$ and $q = \tfrac{p'}{k+1}$ where $k\geq 3$ is the type of $\Sigma$, the following estimate holds:
\begin{equation}\label{cor2a}
\|\hat{F}|_{S}\|_{L^q(S, \ud \sigma)}   \leq A_{p, S} \|F\|_{L^p(\R^3)}  
\end{equation}
for all $F \in \Sc(\R^3)$.
\end{cor}

\begin{remark}\label{rmk} In the appendix it will be shown that the sharp range for which \eqref{cor2a} holds is given by $1 \leq p \leq \tfrac{k+2}{k+1}$ and therefore \eqref{cor2} is almost optimal. It is also remarked that it seems unlikely interpolating some elementary inequality with the result of Proposition \ref{affrest} can produce the endpoint estimate owing to the different kind of behaviour of exhibited by the weighted operator when $p=q$.

\end{remark}

\begin{proof}[Proof (of Corollary \ref{cor2}).] By interpolation with the trivial $(p, q) = (1, \infty)$ estimate, it suffices to show the restricted weak-type version of \eqref{cor2a} holds for all $1 < p < \tfrac{k+2}{k+1}$ and $q = \tfrac{p'}{k+1}$. Fix a pair of exponents $(p,q)$ satisfying these hypotheses and let
\begin{equation*}
\rho := \frac{k-2 -p(k-3)}{k - 1 - p(k-2)} \qquad \tau := \frac{q(k+1) - (k-2)}{3}.
\end{equation*}
It is easy to verify $1 \leq \rho < 4/3$ and $\tau = \rho ' /3$ and so the pair of exponents $(\rho, \tau)$ satisfies the conditions of Corollary \ref{cor1}.  Now partition $\Delta$ into sets $\Delta_j$ defined as follows:
\begin{equation*}
\Delta_j := \big\{ \xi \in \Delta : 2^j \leq w(\xi) < 2^{j+1} \big\}.
\end{equation*}
Fix a measurable subset $E \subset \R^3$ of finite measure and $\alpha > 0$ and consider
\begin{eqnarray*}
\big|\big\{  \xi \in \Delta_j : |\widehat{\chi_{E}}(\xi, \phi(\xi))| > \alpha \big\} \big| &\leq& \frac{1}{\alpha^{\tau}} \int_{\Delta_j}  |\widehat{\chi_{E}}(\xi, \phi(\xi))|^{\tau} \,\ud \xi \\
 &\leq& \frac{2^{-j/3}}{\alpha^{\tau}} \int_{\Delta_j}  |\widehat{\chi_{E}}(\xi, \phi(\xi))|^{\tau}w(\xi)^{1/3} \,\ud \xi \\
&\lesssim_{p,k}& \bigg(\frac{2^{-j/ 3\tau}}{\alpha} |E|^{1/\rho} \bigg)^{\tau}
\end{eqnarray*}
where the last inequality follows by applying Corollary \ref{cor1}. On the other hand, using the homogeneity of the weight one observes
\begin{eqnarray*}
\big|\big\{  \xi \in \Delta_j : |\widehat{\chi_{E}}(\xi, \phi(\xi))| > \alpha \big\} \big| &\leq& \frac{1}{\alpha} \int_{\Delta_j}  |\widehat{\chi_{E}}(\xi, \phi(\xi))| \,\ud \xi \\
&\lesssim_{\phi}& \frac{1}{\alpha}\sigma (\Sigma_j) |E|
\end{eqnarray*}
for $\Sigma_j = \{\xi' \in \Sigma : 2^j \leq w(\xi') < 2^{j+1}\}$. By applying the sublevel set version of van der Corput's lemma (see, for instance, \cite{Carbery1999}) together with the curvature hypothesis, one may deduce the estimate $\sigma (\Sigma_j) \lesssim _{\phi} 2^{j/(k-2)}$. To conclude the proof note, for suitably chosen $J \in \Z$, 
\begin{eqnarray*}
\big|\big\{  \xi \in \Delta : |\widehat{\chi_{E}}(\xi,\phi(\xi))| > \alpha \big\} \big| &=& \sum_{j = -\infty}^{\infty} \big|\big\{  \xi \in \Delta_j : |\widehat{\chi_{E}}(\xi, \phi(\xi))| > \alpha \big\} \big|  \\
&\lesssim_{p,\phi}& \sum_{j = -\infty}^{\infty} \min\bigg\{ \bigg(\frac{2^{-j/ 3\tau}}{\alpha} |E|^{1/\rho} \bigg)^{\tau}, \frac{ 2^{j/(k-2)}}{\alpha}|E| \bigg\} \\
&\lesssim_{p,k}& \bigg(\frac{2^{-J/ 3\tau}}{\alpha} |E|^{1/\rho} \bigg)^{\tau} + \frac{ 2^{J/(k-2)}}{\alpha}|E| \\
&\lesssim& \frac{1}{\alpha^{3(\tau - 1)/(k+1) +1}} |E|^{3(\tau/\rho -1)/(k+1) + 1}
\end{eqnarray*}
where the last inequality is given by picking $J$ to optimise the estimate. By the definition of the exponents $\rho$ and $\tau$ it follows
\begin{equation*}
\big|\big\{  \xi \in \Delta  : |\widehat{\chi_{E}}(\xi,\phi(\xi))| > \alpha \big\} \big|
\lesssim_{p,\phi} \bigg(\frac{1}{\alpha} |E|^{1/p}\bigg)^q,
\end{equation*}
as required. 
\end{proof}

In the appendix slicing will be applied directly to type $k$ conical surfaces to prove a sharp version of Corollary \ref{cor2}.

\paragraph{2. Conjectured results in higher dimensions.}

Now consider the analogous problem in higher dimensions. Let $\Sigma \subset \widehat{\R}^n$ be a smooth hypersurface, given by the boundary of some centred convex body. As before there exists $\phi: \widehat{\R}^n \rightarrow [0,\infty)$ smooth away from the origin, homogenous of degree 1 and such that $\phi(\xi) = 1$ if and only if $\xi \in \Sigma$. Define the weight $w$ by
\begin{equation*}
w(\xi):= \big\langle M(\phi)(\xi) \nabla \phi(\xi) , \nabla \phi(\xi) \big\rangle \phi(\xi)
\end{equation*}
where $M$ is an $(n-1) \times (n-1)$ matrix-valued function given by the negative of the adjugate of the Hessian matrix of $\phi$. It is not difficult to show if $\kappa$ denotes the Gaussian curvature of $\Sigma$ then
\begin{equation}\label{weightformula}
\kappa(\xi) = \frac{w(\xi)}{|\nabla \phi (\xi) |^{n+1}} \qquad \textrm{for all $\xi \in \Sigma$.}
\end{equation}
By considering the conjectured $L^p-L^q$ bounds for the prototypical case of the light cone $\{ (\xi, |\xi|) : \xi \in \hat{\R}^n\}$ (as described in, for example, \cite{Tao2003a}) one is led to the following conjecture:
\begin{conj} For $1 \leq p < \tfrac{2n}{n+1}$ and $q = \tfrac{n-1}{n+1} p'$ the following holds:
\begin{equation*}
\bigg(\int_{\widehat{\R}^n} |\widehat{F}(\xi, \phi(\xi))|^q w(\xi)^{1/(n+1)} \,\frac{\ud \xi}{\phi(\xi)}\bigg)^{1/q} \leq A_p \|F\|_{L^p(\R^{n+1})}  
\end{equation*}
for all $F \in \Sc(\R^{n+1})$.
\end{conj}

In order to proceed as before one would need an $n$-dimensional analogue of Sj\"olin's theorem. Based on the conjectured results for the restriction operator associated to the $(n-1)$-dimensional sphere in $\widehat{\R}^n$ (see, for example, \cite{Tao2003a}) one is led to the following \emph{affine restriction conjecture}:

\begin{conj}[Affine restriction conjecture] For $\Sigma \subset \widehat{\R}^n$ as above and $1 \leq p <\frac{2n}{n+1}$ and $q =\tfrac{n-1}{n+1}p'$ the following holds:
\begin{equation*}
\|\hat{f}|_{\Sigma}\|_{L^q(\Sigma, \,\kappa^{1/(n+1)}\ud \sigma)} \leq A_p \|f\|_{L^p(\widehat{\R}^n)}
\end{equation*}
for all $f \in \Sc(\R^n)$, where $\kappa$ denotes the Gaussian curvature of $\Sigma$ and $\ud \sigma$ surface measure. 
\end{conj}

One can adapt the proof of Proposition \ref{affrest} to show whenever the affine restriction conjecture holds for some $\Sigma$ and choice of exponents $p,q$ satisfying $\tfrac{p'}{n+1} = \tfrac{q}{n-1}$ and $p' > \tfrac{2n}{n-1}$, the estimate for the corresponding cone holds for the same pair of exponents. 

It is remarked that a number of partial results are known regarding the affine restriction conjecture. The majority of these pertain to surfaces of revolution in $\R^3$: for affine restriction in this special case and related results see \cite{Abi-Khuzam2006, Carbery2007, Oberlin2004, Shayya2007, Shayya2009}. Recently, Oberlin \cite{Oberlin2012} proved an affine restriction theorem for hypersurfaces in $\R^n$ under a weak `multiplicity condition'. In \cite{Carbery2002} Carbery and Ziesler discuss the possibility of \emph{universal} affine restriction estimates in higher dimensions. Finally, interesting connections between the affine restriction conjecture and the affine isoperimetric inequality have been observed and discussed in \cite{Carbery2002, Shayya2007, Shayya2009}. 

\paragraph{3. Optimality of the weight.}

Here arguments from \cite{Iosevich2000, Nicola2008} are adapted in order to study the weight function
\begin{equation*}
w(\xi):= \big\langle M(\phi)(\xi) \nabla \phi(\xi) , \nabla \phi(\xi) \big\rangle \phi(\xi).
\end{equation*}
In particular, the following proposition demonstrates that $w$ is a natural choice of weight for the conic restriction problem.

\begin{prop}\label{optimalweight} There exists a constant $c > 0$, independent of $\phi$, such that whenever $0 \leq \psi \in C(\widehat{\R}^n)$ is a weight for which the following conic restriction estimate holds:
\begin{equation}\label{aga}
\bigg(\int_{\widehat{\R}^n} |\hat{F}(\xi, \phi(\xi))|^2 \psi(\xi) \frac{\ud \xi}{\phi(\xi)}\bigg)^{1/2} \leq A\|F\|_{L^p(\R^{n+1})}
\end{equation}
for all $F \in \Sc(\R^{n+1})$ and $p = \tfrac{2(n+1)}{n+3}$, it follows that
\begin{equation}\label{agb}
\psi(\xi) \leq cA^2w(\xi)^{1/(n+1)} \qquad \textrm{for all $\xi \in \widehat{\R}^n\setminus\{0\}$.}
\end{equation}
\end{prop} 

To prove the proposition points where the curvature of $\Sigma$ vanish are considered separately.

\begin{lemma} Suppose $0 \leq \psi$ is a continuous real-valued function on $\widehat{\R}^n$ for which the following restriction estimate holds:
\begin{equation*}
\bigg(\int_{\widehat{\R}^n} |\hat{F}(\xi, \phi(\xi))|^2 \psi(\xi) \,\frac{\ud \xi}{\phi(\xi)} \bigg)^{1/2} \leq A \|F\|_{L^p(\R^{n+1})}
\end{equation*}  
for all $F \in \Sc(\R^{n+1})$ and $p = \tfrac{2(n+1)}{n+3}$. If $\psi(\xi_0) > 0$ for some $\xi_0 \in \Sigma$ then the curvature of $\Sigma$ does not vanish at $\xi_0$. 
\end{lemma}

\begin{proof} The proof is a minor adaptation of the work of Iosevich and Lu in \cite{Iosevich2000}. By rotating the problem one may assume $\xi_0 = (0, |\xi_0|)$ lies on the positive $\xi_n$-axis. In a neighbourhood of $\xi_0$ the surface $\Sigma$ is given by the graph of a smooth function $\gamma: U \rightarrow \R$ where $U:= B(0, \epsilon) \subset \R^{n-1}$ is an open ball about the origin of radius $0 < \epsilon < 1$ and $\gamma(0) = |\xi_0|$. Furthermore, by choosing $\epsilon$ sufficiently small one may assume $\psi(tu, t\gamma(u)) \gtrsim_{\psi} 1$ for all $u \in U$, $t \in (1 - \epsilon, 1+ \epsilon)$ and also, by rotating the co-ordinate space, that the Hessian matrix $\gamma''(0)$ of $\gamma$ at 0 is diagonal.
Applying the co-area formula one observes\footnote{For notational convenience, throughout this proof the dependence of constants upon $\phi$, $\psi$, $A$ and $\gamma$ is suppressed by writing $\lesssim$ and $\gtrsim$ rather than $\lesssim_{\phi, \psi, \gamma}$ and $\gtrsim_{\phi, \psi, \gamma}$, respectively.}
\begin{equation}\label{IL1}
\int_{\widehat{\R}^n} |\hat{F}(\xi, \phi(\xi))|^2 \psi(\xi) \,\frac{\ud \xi}{\phi(\xi)} \gtrsim \int_{1-\epsilon}^{1+\epsilon} \int_U |\hat{F}(tu, t\gamma(u), t)|^2 \,\ud u \ud t.
\end{equation}
This is due to the fact that for all $u \in U$ and $t \in (1-\epsilon, 1+\epsilon)$,
\begin{equation*}
\psi(tu, t\gamma(u)) \frac{(1 + |\nabla\gamma(u)|^2)^{1/2}}{|\nabla\phi(u, \gamma(u))|}t^{n-2} \gtrsim 1. 
\end{equation*}
Fix some $0 < \delta \ll 1$ and define the anisotropic dilations $\underline{\delta}^{\pm 1}$ by
\begin{equation*}
\underline{\delta}^{\pm 1}\underline{\xi} = (\delta^{\pm3/2}\xi_1, \dots, \delta^{\pm3/2}\xi_{n-2}, \delta^{\pm 1}\xi_{n-1}) \qquad \textrm{for $\underline{\xi} = (\xi_1, \dots, \xi_{n-1}) \in \widehat{\R}^{n-1}$}
\end{equation*}
and also let $G(\underline{\delta}) = \sup\{|\Gamma(\underline{\delta}u)| : u \in U \}$ where
\begin{equation*}
\Gamma(u) = \gamma(u) - \gamma(0) - u.\nabla\gamma(0) \quad \textrm{for $u \in U$.} 
\end{equation*}
Fix $g_1 \in \Sc(\widehat{\R}^{n-1})$ and $g_2, g_3 \in \Sc(\widehat{\R})$ with $g_1(\underline{\xi}) = g_2(\xi_n) = g_3(\eta) = 1$ for all $\underline{\xi} \in (1+\epsilon)U$, $\xi_n \in [-1-\epsilon, 1+ \epsilon]$ and $\eta \in [1- \epsilon, 1+\epsilon]$. Let $F_{\delta} \in \Sc(\R^{n+1})$ satisfy
\begin{equation*}
\hat{F}_{\delta}(\underline{\xi}, \xi_n, \eta) = g_1(\underline{\delta}^{-1} \underline{\xi})g_2\Big( \frac{\xi_n - \eta \gamma(0) - \underline{\xi}.\nabla\gamma(0)}{G(\underline{\delta})}\Big) g_3( \eta)
\end{equation*}
for $(\underline{\xi}, \xi_n, \eta) \in \widehat{\R}^{n-1} \times \widehat{\R}\times \widehat{\R}$. Substituting $F_{\delta}$ for $F$ in \eqref{IL1}, the right-hand integral is then equal to
\begin{equation*}
\delta^{(3n-4)/2} \int_{1-\epsilon}^{1+\epsilon} \int_{\underline{\delta}^{-1}U} \bigg|g_1(tu)g_2\bigg(\frac{t\Gamma(\underline{\delta}u)}{G(\underline{\delta})}\bigg)g_3(t)\bigg|^2 \,\ud u \ud t.
\end{equation*}
The double integral is bounded below by
\begin{equation*}
\bigg|\bigg\{ (u,t) \in U \times (1-\epsilon, 1+\epsilon) : tu \in (1 + \epsilon) U, \,\frac{t\Gamma(\underline{\delta}u)}{G(\underline{\delta})} \in [-1-\epsilon, 1+ \epsilon] \bigg\}\bigg|
\end{equation*}
and it is easy to see that the set appearing in the preceding expression contains
\begin{equation*}
\bigg\{ (u,t) \in U \times (1-\epsilon, 1+\epsilon) : \frac{\Gamma(\underline{\delta}u)}{G(\underline{\delta})} \in [-1, 1] \bigg\}
\end{equation*}
and so has measure $\gtrsim 1$. On the other hand, for $p = \tfrac{2(n+1)}{n+3}$ one observes
\begin{equation*}
\|F_{\delta}\|_{L^p(\R^{n+1})} = \big(\delta^{(3n-4)/2} G(\underline{\delta})\big)^{1/p'} \|\check{g_1}\|_{L^p(\R^{n-1})}\|\check{g_2}\|_{L^p(\R)}\|\check{g_3}\|_{L^p(\R)},
\end{equation*}
so that the restriction estimate implies
\begin{equation*}
\delta^{(3n-4)/2} \lesssim G(\underline{\delta})^{(n-1)/2}.
\end{equation*}
The remainder of the argument is identical to the proof of Theorem 2 of \cite{Iosevich2000}. If $1 \leq k \leq n-1$ is the number of non-vanishing principal curvatures of $\Sigma$ then, without loss of generality (by relabelling the variables),
\begin{equation*}
\Gamma(u) = \sum_{j=1}^k a_j u_j^2 + R(u) \qquad \textrm{for $u = (u_1, \dots, u_{n-1}) \in U$} 
\end{equation*}
where $a_1, \dots, a_k \neq 0$ and $R$ is a higher order remainder term. The result follows if $k = n-1$ so assume $k < n-1$. In this case, notice $|R(\underline{\delta} u)| \lesssim \delta^3$ for all $u \in U$ and it follows
\begin{equation*}
G(\underline{\delta})^{(n-1)/2} \lesssim \bigg(\sum_{j=1}^k |a_j| \delta^3 + \delta^3\bigg)^{(n-1)/2} \lesssim \delta^{(3n - 3)/2}.
\end{equation*}
Thus one obtains
\begin{equation*}
\delta^{(3n-4)/2} \lesssim \delta^{(3n - 3)/2} \qquad \textrm{for all $0 < \delta \ll 1$}
\end{equation*}
and this is the desired contradiction.
\end{proof}

The proof of the proposition can now be given. 

\begin{proof}[Proof (of Proposition \ref{optimalweight})] Throughout the proof for any $\xi \in \widehat{\R}^n$, $\xi'$ will denote the unique point in $\Sigma$ such that $\xi = \phi(\xi)\xi'$.

Observe the estimate \eqref{aga} is \emph{scale-invariant} in the sense that it implies the same estimate but with $\psi(\xi)$ replaced by the dilate $\psi(R\xi)$ for any $R > 0$. Fix $\xi_0 \in \widehat{\R}^n$ and note by rotating the problem one may assume $\xi_0 = (0, |\xi_0|)$ and by scale-invariance it suffices to prove the inequality 
\begin{equation}\label{agbr}
\psi(R \xi_0) \leq cA^2w(\xi_0)^{1/(n+1)} 
\end{equation}
for any $R >0$. Hence one may assume $1 \leq \phi(\xi_0) \leq 2$ with $\psi(\xi_0) = \inf \{ \psi(t\xi_0') : 1 \leq t \leq 2\}$ and, by the previous lemma, that the curvature of $\Sigma$ is strictly positive at $\xi_0' \in \Sigma$.

Clearly \eqref{aga} implies the estimate
\begin{equation*}
\bigg(\int_{\Delta} |\hat{F}(\xi, \phi(\xi))|^2 \tilde{\psi}(\xi) \,\frac{\ud \xi}{\phi(\xi)}\bigg)^{1/2} \leq A\|F\|_{L^p(\R^{n+1})}
\end{equation*}
for all $F \in \Sc(\R^{n+1})$ where $\Delta := \{ \xi \in \widehat{\R}^n : 1 \leq \phi(\xi) \leq 2\}$ and 
\begin{equation*}
\tilde{\psi}(\xi) = \inf \{ \psi(t\xi') : 1 \leq t \leq 2\} \qquad \textrm{for $\xi \in \widehat{\R}^n\setminus\{0\}$.}
\end{equation*}
The new weight $\tilde{\psi}$ is continous, satisfies $\tilde{\psi}(\xi_0') = \psi(\xi_0)$ and has the advantage of being homogeneous of degree 0. 

As in the proof of the preceding lemma, in a neighbourhood of $\xi_0'$ the curve $\Sigma$ is given by the graph of a smooth function $\gamma : U \rightarrow \R$ for $U \subset \R^{n-1}$ an open ball about 0 with $\gamma(0) = |\xi_0|$ and, by the curvature hypothesis, $\det \gamma''(0) \neq 0$. Once again one may assume $\gamma''(0) = \textrm{diag}[a_1, \dots, a_{n-1}]$; that is, $\gamma''(0)$ is a diagonal matrix with entries $\gamma''(0)_{ii} = a_i$ for $i =1, \dots, n-1$. Notice there is a natural way to parametrise the dilates of $\Sigma$: for any $t > 0$, in a neighbourhood of $t\xi'_0$ the curve $\Sigma_t$ is given by the graph of the function $\gamma_t : tU \rightarrow \R$  where
\begin{equation*}
\gamma_t(u) = t\gamma(t^{-1}u). 
\end{equation*}
Apply the co-area formula to obtain
\begin{eqnarray*}
\int_{\Delta} |\hat{F}(\xi, \phi(\xi))|^2 \tilde{\psi}(\xi)\frac{\ud \xi}{\phi(\xi)} &=& \int_1^2 \int_{\Sigma_t} |\hat{F}(\xi, t)|^2 \tilde{\psi}(\xi) \frac{\ud\sigma_t(\xi)}{|\nabla \phi(\xi)|}\,\frac{\ud t}{t} \\
&\geq& \int_1^2 \int_{tU} |\hat{F}(u, \gamma_t(u), t)|^2 \Psi(u,t)^2 \ud u \,\frac{\ud t}{t}
\end{eqnarray*}
where
\begin{equation*}
\Psi(u,t)^2 = \chi(u,t)(1 + |\nabla \gamma_t (u)|^2)^{1/2} \frac{\tilde{\psi}(u, \gamma_t(u))}{|\nabla \phi(u, \gamma_t(u))|}
\end{equation*}
for a suitable cut-off function $\chi$ on $\R^{n-1} \times \R$. Thus one concludes
\begin{equation}\label{ag1}
\bigg(\int_{\Delta} |\hat{F}(\xi, \phi(\xi))|^2 \tilde{\psi}(\xi)\frac{\ud \xi}{\phi(\xi)}\bigg)^{1/2} \geq 2^{-1/2} \big\|\Psi(u,t) \hat{F}(u, \gamma_t(u), t)\big\|_{L^2_{u,t}(\R^{n-1} \times [1,2])}.
\end{equation}

The remainder of the proof is an adaptation of the method used to establish Theorem 1.1 of \cite{Nicola2008}. Let $f_1 \in \Sc(\widehat{\R}^{n-1})$ and $f_2, f_3 \in \Sc(\widehat{\R})$; fix $\delta > 0$ and consider the function $F_{\delta} \in \Sc(\R^{n+1})$ defined by
\begin{equation*}
\hat{F_{\delta}}(\underline{\xi}, \xi_n, \eta) = f_1\big(\delta^{-1} |\gamma''(0)|^{1/2}\underline{\xi}\big)f_2\big(\delta^{-2}(\xi_n - \eta\gamma(0) - \underline{\xi}.\nabla\gamma(0))\big)f_3(\eta)
 \end{equation*}
where the notation $|\gamma''(0)|^{\pm 1/2} := \textrm{diag}[|a_1|^{\pm1/2},\dots,|a_{n-1}|^{\pm1/2}]$ has been introduced. Define
\begin{equation*}
\Gamma_t(u) =\gamma_t(u) - \gamma_t(0) - u. \nabla\gamma_t(0) \qquad \textrm{ for $u \in U$.}
\end{equation*}
Now take $F$ to be $F_{\delta}$ in \eqref{ag1} and apply the hypothesised inequality together with a change of variables to deduce
\begin{align*}
\big\|\Psi\big(\delta|\gamma''(0)|^{-1/2}u,t\big) f_1(u)f_2\big(\delta^{-2}\Gamma_t\big(\delta|\gamma''(0)|^{-1/2}u\big)\big)f_3(t) \big\|_{L^2_{u,t}(\R^{n-1} \times [1,2])} \\
\leq 2^{1/2}A |\det\gamma''(0)|^{1/2(n+1)} \|\check{f}_1\|_{L^p(\R^{n-1})}\|\check{f}_2\|_{L^p(\R)}\|\check{f}_3\|_{L^p(\R)}.
\end{align*}
Finally, utilising the homogeneity of both $\tilde{\psi}$ and $\nabla \phi$ one observes
\begin{equation*}
\lim_{\delta \rightarrow 0} \Psi\big(\delta|\gamma''(0)|^{-1/2}u,t\big) = \big(1 +|\nabla\gamma(0)|^2\big)^{1/4}\bigg(\frac{\psi(\xi_0)}{|\nabla \phi(\xi_0')|}\bigg)^{1/2}
\end{equation*}
whilst, by applying Taylor's theorem,
\begin{equation*}
\lim_{\delta \rightarrow 0} \delta^{-2}\Gamma_t\big(\delta|\gamma''(0)|^{-1/2}u\big) =\sum_{j=1}^{n-1}  \frac{\mathrm{sgn}\,a_j}{2t} u_j^2 \qquad u = (u_1, \dots, u_{n-1}) \in U.
\end{equation*}
Hence, by the formula for the curvature of a graph-parametrised hypersurface together with \eqref{weightformula}, one concludes
\begin{eqnarray*}
\psi(\xi_0) &\leq& cA^2 \bigg(\frac{|\det \gamma''(0)||\nabla \phi(\xi_0')|^{n+1}}{(1 +|\nabla\gamma(0)|^2)^{(n+1)/2}} \bigg)^{1/(n+1)} \\
&=& cA^2 w(\xi_0')^{1/(n+1)} = cA^2 w(\xi_0)^{1/(n+1)}
\end{eqnarray*}
where the final equality is due to the homogeneity of $w$. Observe the constant $c$ is given by
\begin{equation*}
c^{1/2} = 2^{1/2} \inf \left\{ \frac{ \|\check{f}_1\|_{L^p(\R^{n-1})}\|\check{f}_2\|_{L^p(\R)}\|\check{f}_3\|_{L^p(\R)}}{\bigg\|f_1(u)f_2\bigg(\sum_{j=1}^{n-1} \frac{\mathrm{sgn}\,a_j}{2t} u_j^2\bigg)f_3(t) \bigg\|_{L^2_{u,t}(\R^{n-1} \times [1,2])}} \right\}
\end{equation*}
where the infimum is taken over all $f_1 \in \Sc(\R^{n-1})$ and $f_j \in \Sc(\R)$ for $j =2,3$.
\end{proof}

\paragraph{4. Appendix: Type $k$ conic surfaces.} To conclude two slight errors in the existing literature, alluded to earlier in the introduction, are highlighted.
\begin{itemize}
\item[i)] Theorem 1 of \cite{Barcelo1986} contradicts the main theorem of \cite{Sogge1987}. For the case $k \geq 3$, the range of exponents in the former should be $1 \leq q \leq p'/(k+1)$ and  $p' \geq k+2$. This discrepancy appears to be due to an incorrect application of the Marcinkiewicz interpolation theorem in \cite{Barcelo1986}. By carrying out the interpolation correctly the proof appears to yield a result agreeing with \cite{Sogge1987}.

\item[ii)] The statement of Corollary 1 of \cite{Barcelo1986} contains a typographical error. Specifically, the correct range of exponents for which inequality \cite[(26)] {Barcelo1986} holds is also $1 \leq q \leq p'/(k+1)$ and  $p' \geq k+2$. Note Barcel\'o's method produces the sub-optimal range $1 \leq q \leq p'/(k+1)$ and  $p' \geq 2k$.

\end{itemize}

The following theorem provides a sharp version of Corollary 1 of \cite{Barcelo1986}.

\begin{thm}\label{correction} Suppose $\Sigma$ is of finite type, let $S$ be as above and $\ud \sigma$ denote surface measure on $S$. For $1 \leq p < \infty$ and $q = p'/(k+1)$ where $k\geq 3$ is the type of $\Sigma$, the following estimate holds:
\begin{equation}\label{app1}
\|\hat{F}|_{S}\|_{L^{q,p}(S, \ud \sigma)}   \leq A_{p} \|F\|_{L^p(\R^3)}  
\end{equation}
for all $F \in \Sc(\R^3)$. The inequality is sharp in the sense that the Lorentz space $L^{q,p}(S, \ud \sigma)$ cannot be replaced with $L^{q,r}(S, \ud \sigma)$ for any $r < p$.
\end{thm}

Note that statement of Theorem \ref{correction} mirrors precisely that of Sogge's restriction theorem for degenerate curves \cite{Sogge1987}. This is what one would expect since, in principle, the behaviour of conic restriction operator should match that of the operator associated to its generating curve.  To prove the inequality \eqref{app1} one can apply Nicola's slicing argument, in conjunction with the aforementioned result of Sogge \cite{Sogge1987}. Indeed, slightly modifying the slicing argument, this time using Lebesgue rather than Lorentz space inequalities, one may deduce 
\begin{equation*}
\|(u \ud \mu_{\mathcal{C}})\,\check{}\,\|_{L^{p'}_{x,t}(\R^3)} \lesssim \big\|s^{1-3/q'}s^{1/q'} \|u(s \,\cdot \,,s)\|_{L^{q',p'}_{\xi}\big(\Sigma, \,\tfrac{\ud \sigma}{|\nabla \phi|}\big)}\big\|_{L_s^p((1,2))}
\end{equation*}
for all suitable $u$ and any pair $(p,q)$ satisfying the hypotheses of Theorem \ref{correction}. Note by real interpolation it suffices to show restricted strong-type inequalities for all such $(p,q)$ and so one may assume $u$ is a characteristic function. The inner Lorentz norm can therefore be replaced with a Lebesgue norm and the proof is concluded by applying H\"older's inequality. See also Theorem 1.3 of \cite{Nicola2009}.

The slicing method also yields the analogous result for the whole cone (with suitably chosen measure) but on the restricted range $1 \leq p \leq \tfrac{k+1}{k+2}$ (so that $q' \leq p'$).  Unfortunately, Nicola's argument does not appear to adapt to give the Lorentz estimates for $q' > p'$ on the whole cone.

It remains to substantiate the claim that the range of $p$ stated in Theorem \ref{correction} is sharp. This is achieved by a minor adaptation of Sogge's counter-example from \cite{Sogge1987}. Fix a conical surface of type $k \geq 3$ and exponents $1 < p < \infty$ and $q = p'/(k+1)$ (the case $p=1$ is simpler and follows from a minor adaptation of the present argument). By rotating the problem and choosing the test function to be supported on a sufficiently small section of the cone it suffices to show for small $0 < \delta$ there exists an integrable function $f:  (1,2) \times (0, \delta) \rightarrow \C$ such that
\begin{equation}\label{app2}
\|f\|_{L^{q', r}((0, \delta) \times (1,2))} < \infty \qquad \textrm{for all $r > p'$}
\end{equation} 
but $\|Tf\|_{L^{p'}(\R^3)} = \infty$ where
\begin{equation*}
Tf(x,y,r) = \int_1^2 \int_0^{\delta} e^{2\pi i s(xt + y \gamma(t) + r)} f(s,t) \,\ud t \ud s
\end{equation*}
for $\gamma: (-\delta, \delta) \rightarrow \R_+$ smooth with $\gamma''(0) = \dots = \gamma^{(k-1)}(0) = 0$ and $\gamma^{(k)}(0) < 0$. To do this simply choose
\begin{equation*}
f(s,t) := t^{-1/q'}|\log t |^{-1/p'} \chi_{[1, 1+ \epsilon]}(s) \qquad \textrm{for all $(s,t) \in (1,2) \times (0, \delta)$}
\end{equation*}
where $0 < \epsilon$ is a small constant to be chosen later. Clearly $f$ satisfies \eqref{app2} whilst, by Fubini's theorem and an obvious change of variables, $\|Tf\|^{p'}_{L^{p'}(\R^3)}$ may be written as
\begin{equation*}
\iiint_{\R^3} \bigg|\int_1^{ 1+ \epsilon} \int_0^{\delta} e^{2\pi i s(xt + y(\gamma(t) - \gamma(0) - t\gamma'(0)) + r)} f(s,t) \,\ud t \ud s \bigg|^{p'} \,\ud x  \ud y \ud r.
\end{equation*}
 Restrict the range of integration in the $(x,y,r)$ variables to $\R^3_+$ and perform the change of variables $(x,y,r) \mapsto (u, \alpha u^k, r)$ to bound the above integral below by
\begin{equation}\label{app3}
\iiint_{\R_+^3} \bigg|\int_1^{ 1+ \epsilon} \int_0^{\delta} e^{2\pi i s(u t + \alpha u^k \Gamma(t))} f(s,t) \,\ud t\,  e^{2\pi i sr} \,\ud s \bigg|^{p'} u^k \,\ud u  \ud \alpha \ud r
\end{equation}
where $\Gamma(t) = \gamma(t) - \gamma(0) - t\gamma'(0)$. By a change of the $t$ variable and the hypotheses on the exponents, the integrand in \eqref{app3} may be rewritten as
\begin{equation*}
\bigg|\int_1^{ 1+ \epsilon} \int_0^{u\delta} e^{2\pi i s(t + \alpha u^k \Gamma(t/u))} t^{-1/q'}|\log t - \log u|^{-1/p'} \,\ud t\,  e^{2\pi i sr} \,\ud s \bigg|^{p'} u^{-1}.
\end{equation*}
Now let $J(u, \alpha, r)$ denote the double integral appearing inside the modulus signs in the preceding expression, multiplied by $|\log u |^{1/p'}$. It is claimed for some choice of $0 \leq \alpha_1 < \alpha_2$, $0 \leq r_1 < r_2$ and $R$ sufficiently large
\begin{equation}\label{app4}
|J(u, \alpha, r)| \geq A > 0 \textrm{ for all $\alpha \in [\alpha_1, \alpha_2]$, $r \in [r_1, r_2]$ and $u \geq R$.}
\end{equation}
Once the claim is established, it follows 
\begin{equation*}
\|Tf\|^{p'}_{L^{p'}(\R^3)} \geq  A^{p'} \int_{r_1}^{r_2}\int_{\alpha_1}^{\alpha_2} \int_R^{\infty} |\log u|^{-1} u^{-1} \,\ud u  \ud \alpha \ud r  = \infty
\end{equation*}
and this concludes the proof. Write
\begin{eqnarray*}
J(u, \alpha, r) &=& \int_1^{ 1+ \epsilon} \int_0^{u\delta} e^{2\pi i s(t + \alpha u^k \Gamma(t/u))} t^{-1/q'}\Big|1 - \frac{\log t}{\log u}\Big|^{-1/p'} \,\ud t\,  e^{2\pi i sr} \,\ud s \\
&=& \int_1^{1+\epsilon}  I(u,\alpha, s) e^{2\pi i sr} \,\ud s
\end{eqnarray*}
where $I(u,\alpha, s)$ is essentially the integral appearing in the statement of Lemma 3 of \cite{Sogge1987}. By applying the limiting arguments found in the proof of the aforementioned lemma one may deduce $I(u,\alpha, s) = g(\alpha, s) + R(u, \alpha, s)$ where
\begin{equation*}
g(\alpha, s) = \int_0^{\infty} e^{2\pi i s(t + c\alpha t^k)} t^{-1/q'} \,\ud t
\end{equation*}
for $c = \tfrac{\gamma^{(k)}(0)}{k!} < 0$ and the remainder term $R(u, \alpha, s) = o(1)$ as $u \rightarrow \infty$ uniformly for all $s \in [1,1+\epsilon]$ and all $\alpha$ belonging to some small closed interval not containing 0. Fixing such an interval $[\alpha_1, \alpha_2]$ where $0 < \alpha_1 < \alpha_2$ with $\alpha_2$ is chosen sufficiently small for the following argument to hold, it suffices to show for some choice of $0 \leq r_1 < r_2$,
\begin{equation*}
\bigg|\int_1^{1+\epsilon}  g(\alpha, s) e^{2\pi i sr} \,\ud s\bigg| \geq A >0 
\end{equation*}
for all $\alpha$ belonging to (some subinterval of) $[\alpha_1, \alpha_2]$ and all $r \in [r_1, r_2]$. It is easy to see $g(\alpha, s)$ is bounded for $(\alpha, s) \in [\alpha_1, \alpha_2] \times [1, 1 + \epsilon]$ and so by writing the above integral as
\begin{equation*}
\int_1^{1+\epsilon}  g(\alpha, s) \,\ud s + \int_1^{1+\epsilon}  g(\alpha, s) (1 - e^{2\pi i sr}) \,\ud s
\end{equation*}
and letting $r_1 = 0$ and $0< r_2$ sufficiently small, it remains to show
\begin{equation}\label{app5}
\bigg| \int_1^{1+\epsilon}  g(\alpha, s) \,\ud s\bigg| \geq A >0 \qquad \textrm{for all $\alpha \in [\alpha_1, \alpha_2]$.}
\end{equation}
To see this one applies techniques from the study of oscillatory integrals to observe
\begin{equation}\label{app6}
|g(\alpha, s)| \gtrsim_{c, k, q} s^{-1/2}\alpha^{-\tfrac{1}{k-1}\big(\tfrac{1}{2} - \tfrac{1}{q'} \big)}
\end{equation}
for all $(\alpha,s) \in [\alpha_1, \alpha_2] \times [1, 1+\epsilon]$. Once this estimate is established, one exploits the continuity of $g$, choosing $\epsilon$ sufficiently small and perhaps passing to a sub-interval of $[\alpha_1, \alpha_2]$, to conclude \eqref{app5}. 

The estimate \eqref{app6} follows from standard arguments. First note by a change of variables, $g(\alpha, s)$ can be written as a constant (depending only on $c$ and $k$) multiple of
\begin{equation*}
\alpha^{-\tfrac{1}{k-1}\big(1 - \tfrac{1}{q'} \big)} \int_{-1}^{\infty} e^{i\lambda \Phi(s)} (s+1)^{-1/q'} \,\ud s
\end{equation*}
where $\Phi(s) = (s+1) - (s+1)^k/k$ and $\lambda$ equals $\alpha^{-1/(k-1)} s$, up to multiplication by a positive constant. The phase function $\Phi$ has a single, non-degenerate critical point at 0. Introduce a bump function $\beta$ with $0 \leq \beta \leq 1$ and $\beta(s) = 0$ for all $|s| \geq 2 \eta$ and $\beta(s)=1$ for all $|s| \leq \eta$ where $0 \leq \eta < 1/2$. By choosing $\eta$ sufficiently small (this choice depends only on $k$) one has
\begin{equation*}
\bigg| \int_{-1}^{\infty} e^{i\lambda \Phi(s)} \beta(s) (s+1)^{-1/q'} \,\ud s \bigg| \gtrsim_{k, q} \lambda^{-1/2}
\end{equation*}
provided $\lambda$ is sufficiently large (which is ensured by choosing $\alpha$ small). Indeed, this may be deduced by applying methods found in \cite[p. 334-337]{Stein1993}. On the other hand, it is not difficult to check
\begin{equation*}
\bigg| \int_{-1}^{0} e^{i\lambda \Phi(s)}(1 -\beta(s)) (s+1)^{-1/q'} \,\ud s \bigg| \lesssim_{k, q} \lambda^{-(1 - 1/q')}
\end{equation*}
whilst
\begin{equation*}
\bigg| \int_{0}^{\infty} e^{i\lambda \Phi(s)}(1 -\beta(s)) (s+1)^{-1/q'} \,\ud s \bigg| \lesssim_{k, q} \lambda^{-1}
\end{equation*}
from which one concludes
\begin{equation*}
\bigg| \int_{-1}^{\infty} e^{i\lambda \Phi(s)} (s+1)^{-1/q'} \,\ud s \bigg| \gtrsim_{k, q} \lambda^{-1/2} \sim_{k, q} s^{-1/2}\alpha^{\tfrac{1}{2(k-1)}}
\end{equation*}
for all $\alpha \in [\alpha_1, \alpha_2]$ and the estimate \eqref{app6} follows.

\bibliography{Reference}
\bibliographystyle{amsplain}

\end{document}